\documentclass[a4paper]{article}

\usepackage{a4wide}
\usepackage{amsmath,amssymb,amsthm}
\usepackage{mathrsfs,bbm}
\usepackage[american]{babel}
\usepackage[dvips]{graphicx}
\usepackage{array,pifont,varioref,framed}

\graphicspath{{./}{figure/}}

\title{On the numerical solution of a boundary integral equation for the exterior Neumann problem on domains with corners}
\author{Luisa Fermo \\
		{\small\it Department of Mathematics and Computer Science} \\
		{\small\it University of Cagliari} \\
		{\small\it Viale Merello 92, 09123 Cagliari, Italy} \\[  5mm]
		Concetta Laurita \\
		{\small\it Department of Mathematics, Computer Science and Economics} \\
		{\small\it University of Basilicata} \\		
		{\small\it Via dell'Ateneo Lucano 10, 85100 Potenza, Italy}
	   }
\date{}

\newcommand{\C}{\mathcal{C}}

\newcommand{\RR}{\mathbb{R}}
\newcommand{\NN}{\mathbb{N}}
\newcommand{\PP}{\mathbb{P}}

\def\C{{\mathcal{C}}}

\theoremstyle{plain}\newtheorem{theorem}{Theorem}[section]
\theoremstyle{remark}
\theoremstyle{plain}\newtheorem{lemma}[theorem]{Lemma}
\theoremstyle{plain}
\theoremstyle{definition}

\begin{document}
\maketitle

\begin{abstract}
The authors propose a Nystr\"om method to approximate the solution of a boundary integral equation connected with the exterior Neumann problem for Laplace's equation on planar domains with corners. They prove the convergence and the stability of the method and show some numerical tests.

\medskip


\medskip

\end{abstract}

\section{Introduction}
Let us consider the exterior Neumann problem for Laplace's equation
\begin{equation}\label{Neumann}
\left\{
  \begin{array}{ll}
    \Delta u(P)=0, & \hbox{$P \in \mathbb{R}^2 \setminus D$}, \\ \vspace{0.2cm}

    \displaystyle \frac{\partial u(P)}{\partial n_P}=f(P), & \hbox{$P \in \Sigma$}, \\ \vspace{0.2cm}

	|u(P)|= \mathcal{O}(|P|^{-1}), & \hbox{$|P| \to \infty$},
  \end{array}
\right.
\end{equation}
where $D$ is a simply connected bounded region in the plane with the boundary $\Sigma$, $n_P$ is the inward normal vector to $\Sigma$ at $P$, and $f$ is a given sufficiently smooth function on $\Sigma$ satisfying $$\int_\Sigma f(P) d \Sigma_P=0. $$

In order to solve (\ref{Neumann}) via boundary integral equations, one can use the Green representation formula for potential functions on exterior regions
\begin{eqnarray}\label{u}
u(A)&=&u(\infty)-\frac{1}{2\pi} \int_{\Sigma} \frac{\partial u(Q)}{\partial n_Q} \log|A-Q| d \Sigma_Q \nonumber \\ & + &\frac{1}{2\pi} \int_{\Sigma} u(Q) \frac{\partial }{\partial n_Q} \log|A-Q| d \Sigma_Q, \quad A \in \mathbb{R}^2 \setminus D,
\end{eqnarray}
getting, by standard arguments (see, for instance, \cite{A}), the boundary integral equation of the second kind
\begin{equation}\label{eq1}
-(2\pi-\Omega(P))u(P)+\int_\Sigma u(Q) \frac{\partial }{\partial n_Q} \log|P-Q| d\Sigma_Q=g(P), \quad P \in \Sigma
\end{equation}
where $\Omega(P)$ denotes the interior angle to $\Sigma$ at $P$ and
\begin{equation}\label{g}
g(P)=\int_\Sigma f(Q) \log|P-Q| d \Sigma_Q.
\end{equation}
Defining the operator
\begin{equation} \label{operatoreK}
(K u)(P)=(-\pi+\Omega(P))u(P)+\int_\Sigma u(Q) \frac{\partial}{\partial\mathbf{n}_Q} [\log{|P-Q|}] d \Sigma_{Q}, \quad P \in \Sigma
\end{equation}
which is a bounded map from $C(\Sigma)$ into $C(\Sigma)$,
one can rewrite equation (\ref{eq1}) in the following more compact operator form
\begin{equation} \label{equ_operator}
(-\pi+K)u=g.
\end{equation}
This is a boundary integral equation of direct type that, differently from those of indirect type obtained
using the potential theory approach,  presents the advantage of getting the solution  on the boundary without any further calculations. On the other hand,  it involves the integral term (\ref{g}) which is difficult to handle numerically, because of the presence of the logarithmic kernel.

We suppose that the boundary $\Sigma$ is twice continuously differentiable with the exception of corners at points $P_1,\ldots,P_n$ with interior angles
$$\alpha_k=(1-\chi_k)\pi, \quad -1<\chi_k<1, \ \chi_k \neq 0, \quad k=1,\ldots,n.$$
We remark that under this assumption, the operator $K$ is not compact. However it is possible to prove (see, for instance, \cite{AdH,ChandlerGraham}) that, if $(-\pi+K)$ is injective, then the inverse operator $(-\pi+K)^{-1}:C(\Sigma) \to C(\Sigma)$ exists and is bounded.

Several integral equation methods like collocation, Galerkin and Nystr\"om methods,
 having the purpose of approximating the solution of the Dirichlet or the Neumann problem in planar domains with corners,
are available in literature.
Most of them are based on the representation of the solution $u$ in the form of a single or double layer potential and on the resolution of the corresponding boundary integral equations
defined on piecewise smooth curves.

A variety of these methods  makes use of  piecewise polynomial appro\-xi\-mations on graded meshes
\cite{Chandler,Jeon, Kress, Rathsfeld} which if, on the one hand, allow to
achieve arbitrarily high order of convergence, on the other hand, could produce ill-conditioned linear systems as the local degree increases.

Sometimes \cite{Kress,MS} such approaches  are combined with smoothing strategies whose purpose is to improve the rate of convergence of the proposed numerical procedure.

More recently, an extensive literature on efficient numerical methods to discretize boundary integral equations connected with elliptic problems
on domains with corners has been developed (see \cite{Bremer2012_1,Bremer2012_2, BremerRokhlin, BremerRokhlinSammis, BrunoOvalTurc, Helsing2011, HelsingOjala} and the references therein).

For instance, in \cite{BrunoOvalTurc} a new algorithm for the solution of the Neumann problem for the Laplace equation is described. Since the solution of the corresponding boundary integral equation can be unbounded at the corners of the domain, the method proposes the analytical subtraction of singularities in order to get high accuracy and, also, a special treatment of nearly non-integrable integrands in such a way to avoid cancellation errors.

Methods of Nystr\"om  type based on discretization techniques \cite{BremerRokhlin, BremerRokhlinSammis}, as well as  compression and preconditioning schemes for the arising linear systems are proposed in \cite{Bremer2012_1,Bremer2012_2}. Such procedures allow one to produce well-conditioned linear systems that do not become too large for domains with piecewise smooth boundaries having a great number of corners.
Unfortunately, stability and convergence results have not been proved theoretically but only supported by numerical evidence. In fact several numerical examples show that high accuracy is achieved in the computation of the solutions.

In this paper, we do not seek the solution  in the form of a potential but we solve the boundary integral equation (\ref{eq1}) computing
directly the harmonic function $u$, at first on the boundary and then, by using (\ref{u}), on the exterior domain. This approach implies that
we have to evaluate integrals of type (\ref{g}), or approximate them if their analytical expression is not known. Then, we also propose a suitable numerical treatment of these integrals, taking into account the presence of their logarithmic kernel, and provide the related error estimate.

The first step of our method is the introduction of a suitable decomposition of the boundary in order to rewrite  equation (\ref{eq1}) as an equivalent system of integral equations. More precisely, we divide each smooth arc of the boun\-dary in three sections, choosing the two-non central parts of a very small length such that they coincide with the straight segments tangent to the curve at the corner points.
Then, we compute its solution by applying a Nystr\"om method based on global approximation on each  section of the boundary (see \cite{MS}).
The method uses, essentially, a Radau quadrature formula, based on different numbers of quadrature knots according to the different lengths of the smooth sections  involved in the adopted decomposition of the boundary.
Nevertheless, in order to be able to establish stability and convergence results, we need to modify slightly  the discrete operator, approximating the operator $K$, around the corners. We remark that this modification is not only theoretical but it is also performed numerically.

A complete analysis of the convergence and the stability of the proposed procedure is conducted, by showing that the method can be applied to any domain $D$, regardless of  the combination of interior angles.
Moreover, it is also proved that the method always leads to solve well-conditioned linear systems without resorting to preconditioning scheme.

The paper is structured as follows. In Section \ref{preliminari} we introduce some functional spaces and quadrature formulas. In Section \ref{the method} we present the method by giving convergence and stability results. Section \ref{proofs} is dedicated to the proofs. Finally, in Section \ref{tests} we show some numerical tests.

\section{Preliminaries}\label{preliminari}
\subsection{Spaces of functions}
Let us denote by $w$  a weight function on $[0,1]$ and define the space $L^p_w$, $1 \leq p < \infty$  as the set of all measurable functions such that
\[\|fw\|_p=\left(\int_{0}^{1} |f(x)w(x)|^p dx \right)^{\frac 1 p}, \quad 1 \leq p <+\infty.\]
Let us also  introduce the  Sobolev-type subspace $W^p_r(w)$ of $L_w^p([0,1])$ defined as follows
\[W^p_r(w)=\{f \in L_w^p([0,1]) \ \mid \ \|f\|_{W^p_r(w)}= \|fw\|_p +\|f^{(r)}\varphi^rw\|_p<\infty\},\]
where $r$ is a positive integer and $\varphi(x)= \sqrt{x(1-x)}$. If $w \equiv1$ we simply write $W^p_r$ instead of $W^p_r(w)$.

Finally, as usual, for $r \in \NN \cup \{\infty\}$ we denote by
$C^r([0,1])$ the set of all continuous functions with $r$ continuous derivatives  and we introduce the product space
\[ C^r([0,1])^{^m}=\underbrace{C^r([0,1]) \times C^r([0,1]) \times \ldots \times C^r([0,1])}_{m \ \mathrm{times}},\]
which is complete with the norm
\begin{equation} \label{normdirectproduct}
\|(f_1,f_2,\ldots,f_m)\|_{\infty} =\max_{i=1,2,\ldots,m} \|f_i\|_\infty.
\end{equation}
\subsection{Quadrature rules}
In this subsection we report the quadrature formulas we adopt in the numerical method and we mention some results which will be useful in the sequel.

Denoted by $\{p_m(v^{\alpha,\beta})\}_m$ the sequence of polynomials which are orthonormal on $[0,1]$ with respect to the Jacobi weight $v^{\alpha,\beta}(x)=x^\alpha (1-x)^\beta$,  let  $x_{m,k}^{\alpha,\beta}$, $k=1,\dots,m$, be the zeros of $p_m(v^{\alpha,\beta})$ and  $l_k^{\alpha,\beta}$, $k=1,\dots,m$, be the fundamental Lagrange polynomials based on these points. Then, according to this notation, the Gauss-Legendre quadrature formula \cite{DR} reads as
\begin{equation}\label{Legendre}
\int_0^1 f(x) dx=\sum_{k=1}^{m} \lambda_{m,k}^L f(x_{m,k}^L)+e_m^L(f),
\end{equation}
where  $x_{m,k}^L = x_{m,k}^{0,0}$,
$\lambda_{m,k}^L=\int_0^1 l_k^{0,0}(x) dx$, $\forall k \in \{1,\dots, m\}$,  and $e_m^L$ is the remainder term, while the Gauss-Radau formula \cite{DR} is given by
\begin{equation}\label{Radau}
\int_0^1 f(x) dx=\sum_{k=0}^{m} \lambda_{m,k}^R f(x_{m,k}^R)+e_m^R(f),
\end{equation}
with $x_{m,0}^R = 0$, $\lambda_{m,0}^R=\frac{1}{(m+1)^2}$,   $x_{m,k}^R = x_{m,k}^{0,1}$,
$\displaystyle \lambda_{m,k}^R=\int_0^1 l_k^{0,1}(x) v^{0,1}(x) dx$, $\forall k\in \{1, \dots, m\}$,  and $e_m^R$  the quadrature error.

In the next theorem we give an  estimate for $e_m^L(f)$ and $e_m^R(f)$.
To this end, we recall the definition of the weighted error of best polynomial approxi\-ma\-tion
\[E_m(f)_{w,p}=\inf_{P_m \in \PP_m} \left\|\left(f-P_m\right)w\right\|_p,\]
where $\PP_m$ is the set of all algebraic polynomials of degree at most $m$.

Moreover, in the following $\C$ denotes a positive constant which may assume different values in different formulas. We write $\C =\C(a,b,\ldots)$ to say that $\C$ is dependent of the parameters $a,b,\ldots.$ and $\C \neq \C(a,b,\ldots)$ to say that $\C$ is independent of them.
Furthermore, if $A,B > 0$ are quantities depending on some parameters, we will write $A \sim B$, if there exists a positive constant $\C$ independent of the parameters of $A$ and $B$, such that
\[ \frac{1}{\C} \leq \frac{A}{B} \leq \C.\]

\begin{theorem}\cite{MMlibro}\label{corquad}
For all $f \in W^1_r$, $r \geq 1$, it results
\begin{equation}\label{corL}
|e_m^L(f)| \leq \frac{\C}{m^r} E_{2m-1-r}\left(f^{(r)}\right)_{\varphi^r,1}
\end{equation}
and
\begin{equation}\label{corR}
|e_m^R(f)| \leq \frac{\C}{m^r} E_{2m-r}\left(f^{(r)}\right)_{\varphi^r,1},
\end{equation}
where $\varphi(x)=\sqrt{x(1-x)}$ and $\C \neq \C(m,f)$.
\end{theorem}

\section{The method} \label{the method}
In this section we describe the numerical method we propose in order to approxi\-mate the solution of problem (\ref{Neumann}). The procedure consists in three steps. The first one is to rewrite equation (\ref{eq1}) as an equivalent system of integral equations by using a suitable decomposition of the boundary. The second step is to solve this system by applying a Nystr\"om type method based on the Gauss-Radau formula (\ref{Radau}), with a number of quadrature knots depending on the length of the involved arc of the boundary. Finally, the solution (\ref{u}) of the exterior Neumann problem is approximated using the results obtained in the previous step.

\subsection{An equivalent system of integral equations}\label{formulation}
By proceeding in  counterclockwise direction, we denote by $\Gamma_k$ and $\Upsilon_k$ two sufficiently small smooth arcs of the boundary $\Sigma$ intersecting at the corner $P_k$. Moreover, we assume that their lengths are chosen so that $\Gamma_k$ and $\Upsilon_k$ essentially coincide with the segments $\tau_k$ and $t_k$, respectively, tangent to the curve $\Sigma$ at $P_k$ in the sense that
\begin{equation}\label{delta1}
\max_{(x,y) \in \Gamma_k} |y-y_{{\tau}_k}| \leq \delta
\end{equation}
and
\begin{equation}\label{delta2}
\max_{(x,y) \in \Upsilon_k} |y-y_{{t}_k}|  \leq \delta
\end{equation}
where $y_{{\tau}_k}$ and $y_{{t}_k}$ are the ordinates of the points with abscissa $x$ on $\tau_k$ and $t_k$, respectively, and $\delta$ is a very small positive number. Then, denoting by $C_k$ the section connecting  $\Upsilon_{k}$ and $\Gamma_{k+1}$, with $\Gamma_{n+1}\equiv \Gamma_1$, and collectively by $\Sigma_1,\dots,\Sigma_{3n}$ all these sections, starting from $\Gamma_1$, we have

\begin{equation}\label{sigma}
\Sigma=\bigcup_{j=1}^{3n} \Sigma_j, \quad \mathrm {with} \quad   \Sigma_j=\left\{
           \begin{array}{ll}
             \Gamma_k, & \hbox{$j \equiv 1 \ (\mathrm {mod} \ 3)$, \quad  $k=\frac{j-1}{3}+1$} \\
             \Upsilon_k, & \hbox{$j \equiv 2 \ (\mathrm{mod} \ 3)$, \quad $k=\frac{j-2}{3}+1$} \\
             C_k, & \hbox{$j \equiv 0 \ (\mathrm{mod} \ 3)$, \quad $k=\frac{j-3}{3}+1$}
           \end{array}.
         \right.
\end{equation}
In this way, equation (\ref{eq1}) is equivalent to
the following system of $3n$ boundary integral equations
\begin{eqnarray}\label{sist1}
(-2 \pi + \Omega(P)) u_i(P)+ \sum_{j=1}^{3n} \int_{\Sigma_j} u_j(Q) \frac{\partial}{\partial \mathbf{n}_Q} [\log{|P-Q|}] d \Sigma_{Q}=g_i(P),  \\
\hspace{5cm} \quad P \in \Sigma_i, \quad i=1,\ldots,3n, \nonumber
\end{eqnarray}
where $u_i$ and $g_i$ denote the restrictions of the functions $u$ and $g$ to the curve $\Sigma_i$, respectively.

Now, in order to transform the above curvilinear $2D$ integrals into 1D integrals, for each arc $\Sigma_i$
we introduce  a parametric representation $\sigma_i$ defined on the interval $[0,1]$
\begin{equation}\label{trasf}
\sigma_i: s \in [0,1] \to (\xi_i(s), \eta_i(s)) \in \Sigma_i,
\end{equation}
with $\sigma_i \in C^{2}([0,1])$ and $|\sigma_i'(s)|\neq 0$ for each $0 \leq s \leq 1$ and $i=1,\ldots,3n$. Without any loss of generality, we can assume that $\sigma'_i(s)<0$ if $i \equiv 1$(mod 3), $\sigma'_i(s)>0$ if $i \equiv 2,3$(mod 3) and $\sigma_i(0)=P_k$ if $\Sigma_i=\Gamma_k$ or $\Sigma_i=\Upsilon_k$.
Hence, system (\ref{sist1}) becomes, for $s \in [0,1]$,
\begin{equation}\label{sist01}
(-2 \pi + \bar{\Omega}_i(s)) \bar{u}_i(s)+ \sum_{j=1}^{3n} \int_0^1 K^{i,j}(t,s) \bar{u}_i(t) dt= \bar{g}_i(s), \quad i=1,\ldots,3n,
\end{equation}
where $\bar{\Omega}_i(s)= \Omega(\sigma_i(s))$,
$\bar{u}_i(s)=u_i(\sigma_i(s))$, $\bar{g}_i(s)=g_i(\sigma_i(s))$ and
\begin{equation} \nonumber
K^{i,j}(t,s)=\left\{
    \begin{array}{ll}
      \displaystyle \frac{\eta'_j(t)[\xi_i(s)-\xi_j(t)]-\xi'_j(t)[\eta_i(s)-\eta_j(t)]}{[\xi_i(s)-\xi_j(t)]^2+[\eta_i(s)-\eta_j(t)]^2}, & \hbox{$i \neq j \quad \mathrm{or} \quad t \neq s$} \\ \\
    \displaystyle \frac 1 2 \frac{\eta'_j(t) \xi''_j(t)-\xi_j'(t) \eta''_j(t)}{[\xi'_j(t)]^2+[\eta'_j(t)]^2}, & \hbox{$i=j \quad \mathrm{and} \quad t=s$}
    \end{array}
  \right..
\end{equation}
We point out that, in virtue of our assumptions,
\begin{equation} \label{angle}
\bar{\Omega}_i(s)=\bar{\Omega}_{i+1}(s)=\left\{
\begin{array}{ll}
\pi,  & \quad 0<s \leq 1\vspace{0.2cm}\\
(1-\chi_k)\pi,  & \quad s=0
\end{array}
\right.
\end{equation}
when $i \equiv 1$(mod 3) and $k=\frac{i-1}{3}+1$, and  $\bar{\Omega}_i(s)=\pi$,  $\forall s \in [0,1]$, if $i \equiv 3$(mod 3).

In order to carry out the numerical treatment of system (\ref{sist01}), we introduce the operators
\begin{equation}\label{opK}
(\mathcal K^{i,j} \rho)(s)=\int_{0}^1 K^{i,j}(t,s) \rho(t) dt, \quad \rho \in C([0,1]),
\end{equation}
which are compact on the space $C([0,1])$, since their kernels are continuous on $[0,1] \times [0,1]$
(see, for instance, \cite{AdH,Kress}), except when $i,j \equiv 1,2 (\mathrm{mod} \ 3)$  and $|i-j|=1$.
In fact, in such cases $\mathcal K^{i,j}$ takes the following form (see \cite{AdH,ChandlerGraham,Jeon})
\begin{equation}\label{Mellin}
(\mathcal K^{i,j} \rho)(s)=(\mathcal L^{i,j} \rho)(s)+(\mathcal M^{i,j}\rho)(s),
\end{equation}
where the integral operators $\mathcal L^{i,j}  $
\begin{equation}\label{opL}
(\mathcal L^{i,j} \rho)(s)= \int_0^1 L^{i,j} (t,s) \rho(t) dt,
\end{equation}
have a Mellin-type kernel given by
$${L}^{i,j} (t,s)=-\frac{s \sin {(\chi_k \pi)}}{s^2+2ts \cos {(\chi_k \pi)}+t^2},$$
with $k=\frac{i-1}{3}+1$ if $i \equiv 1(\mathrm{mod} \ 3)$, and $k=\frac{i-2}{3}+1$ if $i \equiv 2(\mathrm{mod} \ 3)$,
while the integral operators $\mathcal M^{i,j}$
\begin{equation}\label{opM}
(\mathcal M^{i,j}\rho)(s)=\int_0^1 M^{i,j}(t,s)\rho(t)dt, \quad \rho \in C([0,1]),
\end{equation}
have a continuous kernel $ M^{i,j}(t,s)$ on $[0,1]\times[0,1]$.

Thus,  by collecting all the integral operators with continuous kernels
in the following matrix

\begin{eqnarray} \label{matrixK}
\hspace*{-0.22cm} {\mathcal K}=  \left(
                           \begin{array}{lllllll}
                             \mathcal K^{1,1} & \mathcal M^{1,2} & \mathcal K^{1,3} & \cdots & \mathcal K^{1,3n-2} & \mathcal K^{1,3n-1} & \mathcal K^{1,3n} \\
                              \mathcal M^{2,1}&  \mathcal K^{2,2} & \mathcal K^{2,3} & \cdots & \mathcal K^{2,3n-2} & \mathcal K^{2,3n-1} & \mathcal K^{2,3n} \\
                             \mathcal K^{3,1} & \mathcal K^{3,2} & \mathcal K^{3,3} & \ddots & \mathcal K^{3,3n-2} & \mathcal K^{3,3n-1} & \mathcal K^{3,3n} \\
                             \vdots & \vdots & \vdots & \ddots & \vdots & \vdots & \vdots \\
                             \mathcal K^{3n-2,1} & \mathcal K^{3n-2,2} & \mathcal K^{3n-2,3} & \cdots & \mathcal K^{3n-2,3n-2} & \mathcal M^{3n-2,3n-1} & \mathcal K^{3n-2,3n} \\
                             \mathcal K^{3n-1,1} & \mathcal K^{3n-1,2} & \mathcal K^{3n-1,3} & \cdots & \mathcal M^{3n-1,3n-2} & \mathcal K^{3n-1,3n-1} & \mathcal K^{3n-1,3n} \\
                             \mathcal K^{3n,1} & \mathcal K^{3n,2} & \mathcal K^{3n,3} & \cdots & \mathcal K^{3n,3n-2} & \mathcal K^{3n,3n-1} & \mathcal K^{3n,3n} \\
                           \end{array}
                         \right)
\end{eqnarray}
and the Mellin-type integral operators in the following block matrix
\begin{equation} \label{MatrixL}
{\mathcal L}= \small\left( \begin{array}{lllllll}
                             {\mathcal L}^1 &  &  &      \\
                              &   {\mathcal L}^4 &   &  \\
                              &  &   \ddots &      \\
                               &  &  &  {\mathcal L}^{3n-2} \\
                           \end{array}\right),
\end{equation}
with the blocks ${\mathcal L^i}$, for $i=1,\dots,3n$, and $i \equiv 1(\mathrm{mod} \ 3),$ given by
\begin{equation} \label{MatrixLi}
{\mathcal L^i}=\small\left(\begin{array}{lll}
                              (-\pi+\bar{\Omega}_i)I &  \mathcal L^{i,i+1} &   0\\
                               \mathcal  L^{i+1,i}   & (-\pi+\bar{\Omega}_{i+1})I & 0\\
                              0 & 0 & 0\\
                           \end{array}\right),
\end{equation}

we can rewrite
system (\ref{sist01}), in a compact form, as
\begin{equation} \label{sistop}
(-\pi {\mathcal{I}}+ {\mathcal{L}}+{\mathcal{K}}) \bar{u}=\bar{g},
\end{equation}
where \begin{equation}
{\mathcal I} =\small\left(
                           \begin{array}{llll}
                             I & 0 & \cdots &  0 \\
                              0& I & \cdots & 0   \\
                             \vdots &   \cdots & \ddots  &\vdots \\
                             0 &   \cdots & \cdots  & I \\
                           \end{array}
                         \right),
\end{equation}
with $I$ the identity operator on $C([0,1])$, and
\begin{equation} \label{barpsi}
\bar{u}= \left(\bar{u}_1,\bar{u}_2,\ldots,\bar{u}_{3n}\right)^T, \quad \bar{g}= \left(\bar{g}_1,\bar{g}_2,\ldots,\bar{g}_{3n}\right)^T.
\end{equation}
Now, let us introduce the following complete subspace of $C([0,1])^{^{3n}}$
\begin{eqnarray} \label{spacetilde} \small
{\tilde{\mathcal X}}&=&\left\{(f_1,\ldots,f_{3n})^T \in C([0,1])^{^{3n}}  \mid f_i(0)=f_{i+1}(0), \ f_{i+1}(1)=f_{i+2}(0), \nonumber \right.\\  &{} & \hspace{2 cm} \left.  \ f_i(1)=f_{i-1}(1), \forall i\in \{1,\ldots,3n\}, i \equiv 1 \ (\mathrm{mod} \ 3) \right\},
\end{eqnarray}
with $f_0 \equiv f_{3n}$, and the bijective map $\eta:C(\Sigma) \to \tilde{\mathcal X}$ defined as
$$ \eta f=(\bar{f}_1,\ldots,\bar{f}_{3n}), \quad \bar{f}_i(t)=f(\sigma_i(t)), \quad t \in [0,1].$$
We note that the arrays $\bar{u}$ and $\bar{g}$ introduced in (\ref{barpsi}) belong to ${\tilde{\mathcal{X}}}$ and that  the operator
$$(-\pi {\mathcal{I}}+ {\mathcal{L}}+{\mathcal{K}})^{-1}:\tilde{\mathcal X} \to\tilde{\mathcal X} $$ exists and is bounded. This is a consequence of the equality
\begin{equation}\label{propr_eta}
(-\pi {\mathcal{I}}+ {\mathcal{L}}+{\mathcal{K}})=\eta(-\pi+K)\eta^{-1}
\end{equation}
and of the invertibility of the operator $(-\pi+K):C(\Sigma) \to C(\Sigma)$.
However, in order to carry out the analysis of the stability and the convergence  of the numerical procedure we are going to propose, for approximating the solution of  (\ref{sistop}),  let us also introduce the following complete subspace of $C([0,1])^{^{3n}}$
\begin{equation} \label{space}\small
{\mathcal X}=\left\{(f_1,\ldots,f_{3n})^T \in C([0,1])^{^{3n}}  \mid f_i(0)=f_{i+1}(0), \forall i\in \{1,\ldots,3n\}, \ i \equiv 1 \ (\mathrm{mod} \ 3)\right\}
\end{equation}
equipped with the uniform norm defined in (\ref{normdirectproduct}).
Let us observe that  $\tilde{\mathcal X} \subset {\mathcal X}$.

The next result holds true.

\begin{theorem}\label{theoremsolv}
Let $\mathrm{Ker}(-\pi \mathcal{I}+{\mathcal L}+{\mathcal K})=\{0\}$ in the Banach space ${\mathcal X}$. Then system $(\ref{sistop})$ has a unique solution $\bar{u}$ in ${\mathcal X}$ for each given right-hand side $\bar{g} \in {\mathcal X}$. Moreover, if $\bar{g} \in \tilde{\mathcal X}$ then $\bar{u} \in \tilde{\mathcal X}$.
\end{theorem} \vspace{0.2cm}

We remark that the solution $\bar{u}$ has a low smoothness near the corner points. In fact, if we look in detail the smoothness properties of the solution on each section $\Sigma_i$ of the boundary, it results that (see \cite{A,Chandler,Gris} and the re\-fe\-ren\-ces therein)
 \begin{itemize}
\item for $i \equiv 0 (\mathrm{mod} \ 3)$, $\bar{u}_i$ is smooth;
\item for $i \equiv 1 (\mathrm{mod} \ 3)$,  with $k=\displaystyle\frac{i-1}{3}+1$, or $i \equiv 2(\mathrm{mod} \ 3)$,  with $k=\displaystyle\frac{i-2}{3}+1$, we have
\begin{equation} \label{behavioursolution}
\bar{u}_i(t)=O\left(t^{\beta_k}\right), \quad  0 <t \leq 1, \quad \beta_k=\frac{1}{1+|\chi_k|},
\end{equation}
\begin{equation} \label{behavioursolution2}
\bar{u}_i^{(r)}(t) \leq \C t^{\beta_k-r}, \quad  0 <t \leq 1, \quad r=1,2,\ldots .
\end{equation}
\end{itemize}
Note that, being $\frac{1}{2}<\beta_k<1$, the first derivative of the solution has an algebraic singularity in the corner points $P_k$.

\subsection{A  Nystr\"om method}\label{Nystrom}
In order to approximate the solution $\bar{u}$ of (\ref{sistop}), we introduce
the  finite rank operators defined as follows.
When $i\equiv 1,2$(mod 3) and $|i-j| \neq 0,1$ or when $i \equiv 0$(mod 3), let
\begin{equation} \label{Kmij1}
(\mathcal K_m^{i,j} \rho)(s)=
\displaystyle \sum_{h=0}^{\nu_m} \lambda_{\nu_m,h}^{R}  K^{i,j}(x_{\nu_m, h}^R,s) \rho(x_{\nu_m,h}^{R})
\end{equation}
be the discrete operator, approximating $\mathcal K^{i,j}$ in (\ref{opK}), obtained by applying the Radau formula (\ref{Radau}) with
$\nu_m=\nu(m)$ quadrature points. In the remainder cases, let us define the operators
\begin{equation} \label{Kmij2}
(\mathcal K_m^{i,j} \rho)(s)=\sum_{h=0}^{\mu_m} \lambda_{\mu_m,h}^{R}  K^{i,j}(x_{\mu_m, h}^R,s) \rho(x_{\mu_m,h}^{R}),
\end{equation}
\begin{equation} \label{Mmij}
(\mathcal M_m^{i,j} \rho)(s) =\displaystyle  \sum_{h=0}^{\mu_m} \lambda_{\mu_m,h}^{R}  M^{i,j}(x_{\mu_m,h}^R,s) \rho(x_{\mu_m,h}^R),
\end{equation}
and
\begin{equation} \label{Lmij}
(\mathcal L_m^{i,j} \rho)(s)=\displaystyle  \sum_{h=0}^{\mu_m} \lambda_{\mu_m,h}^{R} L^{i,j}(x_{\mu_m,h}^R,s) \rho(x_{\mu_m,h}^R),
\end{equation}
approximating $\mathcal K^{i,j}$, $\mathcal M^{i,j}$ and $\mathcal L^{i,j}$ defined in (\ref{opK}), (\ref{opM}) and (\ref{opL}), respectively, by means of the same quadrature rule with $\mu_m=\mu(m)$ nodes.
We shall choose $\nu_m$ and $\mu_m$ linear functions of $m$ and we shall assume $\mu_m < \nu_m$, according to the different length of the smooth sections $\Sigma_j$ involved in the adopted decomposition of the boundary (for instance, $\nu_m=m$ and $\mu_m$ a fraction of $m$).

At this point, if we apply the Nystr\"om method based on these quadrature formulas,  we should  solve the following approximating system
\begin{equation}
(-\pi \mathcal{I}+{\mathcal L}_m+{\mathcal K}_m)\bar{u}_m=\bar{g},
\end{equation}
where
 ${\mathcal L}_m$ is the matrix  obtained by replacing in (\ref{MatrixL})  the blocks $\mathcal{L}^i$ with
 \begin{equation} \label{Lmi}
{\mathcal L}_m^i=\left(\begin{array}{ccc}
(-\pi+\bar{\Omega}_i)I & \mathcal L_m^{i,i+1} & 0\\
\mathcal L_m^{i+1,i} & (-\pi + \bar{\Omega}_{i+1}) I & 0 \\
0 & 0 & 0
\end{array}\right),
\end{equation}
 ${\mathcal K}_m$ is the matrix given by (\ref{matrixK}) but with the operators $\mathcal K_m^{i,j}$ and $\mathcal M_m^{i,j}$ in place of  $\mathcal{K}^{i,j}$ and  $\mathcal{M}^{i,j}$, respectively,
 and $\bar{u}_m=(\bar{u}_{m,1},\ldots.,\bar{u}_{m,3n})^T$ is the array of the unknowns.
However, proceeding in this way we will not be able to prove the stability and convergence of the method.
Indeed, it is possible to establish (see Theorem \ref{opercompatti}) that any sequence of operators $\{\mathcal K_m^{i,j}\}_m$ and $\{\mathcal M_m^{i,j}\}_m$ are pointwise convergent to the operators $\mathcal K^{i,j}$ and $\mathcal M^{i,j}$, respectively  in the space $C([0,1])$,  whereas we can not state a similar result
for the sequences $\{\mathcal L_m^{i,j}\}_m$.
More precisely, we are able to prove that, for  any $\rho \in C([0,1])$,  each sequence of functions $\left\{ \mathcal L_m^{i,j}\rho\right\}_m$ converges uniformly to $\mathcal L^{i,j}\rho$ in any interval of the type $[\frac{c}{m^{2-2\epsilon}},1]$, for some constant $c>0$ and arbitrarily small $\epsilon>0$, but not in the whole interval $[0,1]$.

To overcome this problem, we propose a perturbated Nystr\"om method based
on a modification of the matrix $\mathcal{L}_m$. Indeed,
following an idea in \cite{FL},  we modify the matrix ${\mathcal L}_m$ by replacing the blocks ${\mathcal L}_m^i$ with the new blocks $\tilde{\mathcal L}_m^i$ defined as \\

$(\tilde{\mathcal L}_m^i \varrho)(s)$
\begin{eqnarray} \label{opmodified}
&=\left\{
  \begin{array}{ll}
    ({\mathcal L}_m^i \varrho)(s), & \hbox{$\displaystyle  \frac{c}{m^{2-2\epsilon}} \leq s \leq 1$}, \\ \\
    \displaystyle \frac{m^{2-2\epsilon}}{c}\left[s({\mathcal L}_m^i \varrho)\left(\frac{c}{m^{2-2\epsilon}}\right)
    +\left(\frac{c}{m^{2-2\epsilon}} -s\right)({\mathcal L}^i \varrho)(0)\right],
   & \hbox{$0 \leq s < \displaystyle  \frac{c}{m^{2-2\epsilon}}$,}
  \end{array} \right.
\end{eqnarray}
for $i=1,\ldots,3n$, $i\equiv 1(\mathrm{mod} \ 3)$,
 $\varrho \in C([0,1])^{^3}$, where $c>0$ is a fixed constant and $\epsilon>0$ is an arbitrarily small number.

Then, denoting by $\tilde {\mathcal L}_m$ the matrix thus obtained, in place of  (\ref{sistop}), we consider the new approximating system
\begin{equation} \label{sistapprox}
(-\pi \mathcal{I}+\tilde {\mathcal L}_m+{\mathcal K}_m)\bar{u}_m=\bar{g}.
\end{equation}

The operators $\tilde{\mathcal L}_m$ and ${\mathcal K}_m$ satisfy the following theorems.
\begin{theorem} \label{operL}
The operators $\tilde{\mathcal  L}_m:{\mathcal X} \to {\mathcal X}$ are linear maps such that
\begin{equation} \label{limnormWm}
 \lim_{m \to \infty} \|\tilde{\mathcal  L}_m \| < \pi
 \end{equation}
 and
 \begin{equation} \label{convpointnormWm}
 \lim_{m \to \infty} \|(\tilde{\mathcal  L}_m -{\mathcal  L})\rho\|_{\infty}=0, \quad \forall \ \rho \in {\mathcal X}.
 \end{equation}
\end{theorem}

\begin{theorem}\label{opercompatti}
The operators  ${\mathcal K}_m:{\mathcal X} \to {\mathcal X}$ are linear maps such that the set $\left\{{\mathcal K_m}\right\}_m$ is collectively compact and
\begin{equation} \label{convpointnormSm}
\displaystyle \lim_{m \to \infty} \|({\mathcal K}_m -{\mathcal K}){\rho}\|_{\infty}=0, \quad \forall \  \rho \in {\mathcal X}.
 \end{equation}
\end{theorem}

Now, in order to compute the solution of the approximating system (\ref{sistapprox}),  we collocate suitably the equations of (\ref{sistapprox}) in  the quadrature nodes, getting an equivalent linear system.
More precisely, if we denote by $\bar{\psi}_m=\left(\bar{\psi}_{m,1},\bar{\psi}_{m,2},\ldots,\bar{\psi}_{m,3n}\right)$  the  array
$$\bar{\psi}_m=(-\pi \mathcal{I}+\tilde {\mathcal L}_m+{\mathcal K}_m)\bar{u}_m,$$
the linear system consists of the following $M_m= n(2\mu_m+\nu_m+3)$ equations
\begin{equation} \label{linerasist}
\left\{
\begin{array} {lllll}
\bar{\psi}_{m,i}\left(x_{\mu_m,\ell}^R\right)&=&\bar{g}_i\left(x_{\mu_m,\ell}^R\right), &\quad \ell=0,\ldots,\mu_m &\quad \mathrm{if} \ i \equiv 1,2 (\mathrm{mod} \ 3)  \\
\bar{\psi}_{m,i}\left(x_{\nu_m,\ell}^R\right)&=&\bar{g}_i\left(x_{\nu_m,\ell}^R\right), &\quad \ell=0,\ldots,\nu_m &\quad \mathrm{if} \ i \equiv 0 (\mathrm{mod} \ 3)
\end{array}
\right.
\end{equation}
for $i=1,\ldots,3n$, in the $M_m$ unknowns
\begin{eqnarray} \label{ail}
a_{i,\ell}= \left\{
\begin{array}{lll}
\bar{u}_{m,i}\left(x_{\mu_m,\ell}^R\right), &\quad \ell=0,\ldots,\mu_m &\quad \mathrm{if} \ i \equiv 1,2 (\mathrm{mod} \ 3)  \\
\bar{u}_{m,i}\left(x_{\nu_m,\ell}^R\right),   &\quad \ell=0,\ldots,\nu_m &\quad \mathrm{if} \ i \equiv 0 (\mathrm{mod} \ 3)
\end{array}
\right.     .
\end{eqnarray}
Denoting by $\mathbf{a}$ the array of the unknowns, by $A_m$ is  the matrix of the coefficients, and by $\mathbf{b}$
the the right-hand side vector,
we can rewrite the system (\ref{linerasist}) in the compact form
\begin{equation} \label{linearsystem2}
A_m\mathbf{a}=\mathbf{b}.
\end{equation}
We remark that it is equivalent, in some sense, to the approximating problem (\ref{sistapprox}). Indeed, let us denote by $\tilde{\RR}^{{M}_m}$  the subspace of $\RR^{M_m}$ containing all the arrays
\begin{eqnarray}
& &\mathbf{c}=\left(c_{1,0},\dots c_{1,\mu_m}, c_{2,0},\dots c_{2,\mu_m}, c_{3,0}, \dots, c_{3,\nu_m}, \dots , \nonumber \right. \\  & & \hspace{0.5cm} \left. c_{3n-2,0},\dots c_{3n-2,\mu_m}, \dots   c_{3n-1,0},\dots c_{3n-1,\mu_m}, \dots  c_{3n,0}, \dots, c_{3n,\nu_m} \right) \nonumber
\end{eqnarray}
such that $c_{i,0}=c_{i+1,0}$, $\forall i \in \{1,\dots,3n\}$, $i\equiv 1(\mathrm{mod} \ 3)$. Then, each solution $\bar{u}_m \in {\mathcal X}$ of (\ref{sistapprox}) furnishes a solution $\mathbf{a}$ of system (\ref{linearsystem2}) belonging to $\tilde{\RR}^{M_m}$. It will be sufficient to evaluate the components of the vector $\bar{u}_m$ at the suitable quadrature nodes of the Radau formula.
Viceversa, if
\begin{eqnarray}
  & &\mathbf{a}=(a_{1,0},\dots a_{1,\mu_m}, a_{2,0},\dots a_{2,\mu_m}, a_{3,0}, \dots, a_{3,\nu_m}, \dots , \nonumber \\  & & \hspace{0.5cm} a_{3n-2,0},\dots a_{3n-2,\mu_m}, \dots   a_{3n-1,0},\dots a_{3n-1,\mu_m}, \dots  a_{3n,0}, \dots, a_{3n,\nu_m} ) \nonumber
\end{eqnarray}
satisfies (\ref{linearsystem2}), then there is a unique $\bar{u}_m \in {\mathcal X}$,
solution of (\ref{sistapprox}), such that the equalities (\ref{ail}) hold true.

Consequently, we can conclude that the operator $-\pi \mathcal{I}+\tilde{\mathcal L}_m+{\mathcal K}_m$ is invertible on the space $\mathcal X$
if and only if the matrix $A_m$ is invertible on $\tilde{\RR}^{M_m}$.

Next theorem contains our main result.
\begin{theorem}\label{maintheorem}
Let $\Sigma\setminus\{P_1,\ldots,P_n\}$ be of class $C^{2}$ and $f \in C^{p}([0,1])$ with  $p$ large enough.
Assume that $\mathrm{Ker}\{-\pi \mathcal{I}+ {\mathcal L}+ {\mathcal K}\}=\{0\}$ in the Banach space ${\mathcal X}$. Then, for  sufficiently large $m$, say $m \geq m_0$,
the operators $-\pi \mathcal{I}+\tilde{ \mathcal L}_m+{\mathcal K}_m$ are invertible and their inverses are uniformly bounded on ${\mathcal X}$.

Moreover, denoting by $\mathrm{cond}(-\pi \mathcal{I}+\tilde{\mathcal L}_m+{\mathcal K}_m)$ the condition number of the operator $-\pi \mathcal{I}+\tilde{\mathcal L}_m+{\mathcal K}_m:{\mathcal X} \rightarrow {\mathcal X}$  and by $\mathrm{cond}(A_m)$ the condition number of the matrix $A_m$  in infinity norm, we have that, for any $m \geq m_0$,
\begin{equation} \label{conditionnumber}
\mathrm{cond}(A_m) \leq \mathrm{cond}(-\pi \mathcal{I}+\tilde {\mathcal L}_m+{\mathcal K}_m) \leq \C
\end{equation}
where $\C \neq \C(m)$.

Furthermore, the solutions $\bar{u}$ and $\bar{u}_m$ of systems $(\ref{sistop})$ and $(\ref{sistapprox})$, respectively, satisfy the following error estimates
\begin{equation}
\|(\bar{u}-\bar{u}_m)(s)\|_{\infty} \leq  \mathcal{C}[\|(\tilde{\mathcal{L}}_m-\mathcal{L})\bar{u}(s)\|_\infty + \|(\mathcal{K}_m-\mathcal{K})\bar{u}(s)\|_\infty
\end{equation}
where
\begin{equation} \label{errorestimate1}
\|(\tilde{\mathcal{L}}_m-\mathcal{L}) \bar{u}(s)\|_\infty \leq \left\{ \begin{array}{ll}
\C \max\left\{\left(\displaystyle \frac{1}{m^{2-2\epsilon}}\right)^{\beta},\displaystyle \frac{1}{m^{1+\epsilon}}\right\}, & s  \in \left[0,\displaystyle \frac {c}{m^{2-2\epsilon}}\right]\\
\displaystyle \frac{\C}{m^2}\frac{1}{s^{\frac{1}{2}}}, & s  \in \left[\displaystyle \frac {c}{m^{2-2\epsilon}},1\right],
\end{array} \right.
\end{equation}
with $\epsilon$  as in  $(\ref{opmodified})$,  $\displaystyle \beta=\min_{k=1,\ldots,n}\beta_k$, with $\beta_k$ as in $(\ref{behavioursolution})$ and $\C \neq \C(m)$.
\end{theorem}

We remark that (see Theorem \ref{opercompatti})
$$\lim_{m}\|({{\mathcal K}}_m-{\mathcal K})\bar{u}(s)\|_{\infty}=0, \quad \forall s\in [0,1]$$
and the rate of convergence depends on the smoothness of the boundary $\Sigma\setminus\{P_1,\dots,P_n\}$ as well as on the behavior of the functions
 $\bar{u}_j$ on the interval $[0,1]$ (see (\ref{behavioursolution}),
(\ref{behavioursolution2})).

Moreover, we note that the previous theorem establishes the convergence of the approximate solution $\bar{u}_m$ to the exact one $\bar{u}$ in the space $\mathcal{X}$ (and not $\tilde{\mathcal{X}}$). Therefore, by reconstructing the approximate solution $u_m$ on the initial boundary $\Sigma$, it can get a finite number of discontinuity points. Nevertheless, these discontinuities do not play any role when we replace the harmonic function $u$ on $\Sigma$ with
$u_m$ in (\ref{u}) (see also (\ref{double2})),  in order to approximate the solution $u$ of the Neumann problem at points of the exterior domain.

At this point, we investigate on  a possible approximation of the right-hand side $\bar{g}$ of system (\ref{sistapprox}) in the case when it cannot be evaluated analytically.

To this end we decompose the whole boundary $\Sigma$ in $n$ smooth arcs $\tilde{\Sigma}_1,\ldots,\tilde{\Sigma}_n$, with $\tilde \Sigma_k$ connecting the corner point $P_k$ with $P_{k+1}$ ($P_{n+1} \equiv P_1$) and represented by the parametrization $\tilde \sigma_k \in C^{2}([0,1])$
\begin{equation}\label{trasf_sigma}
\tilde \sigma_k: s \in [0,1] \to (\tilde \xi_k(s), \tilde \eta_k(s)) \in \tilde \Sigma_k,
\end{equation}
such that $|\tilde \sigma_k'(s)|\neq 0$ for each $0 \leq s \leq 1$.
Then since, for any fixed $s \in [0,1]$ and $i=1,\ldots,3n$, we can write
\begin{equation} \label{gsigma}
\bar{g}_i(s)=g(\sigma_i(s))=g(\tilde \sigma_\ell(s_i)),
\end{equation}
for some $\ell=1,\ldots,n$ and with a suitable $s_i\in [0,1]$, we focus our attention on the numerical computation of
\begin{equation}\label{gi}
g(\tilde \sigma_\ell(s))= \sum_{k=1}^n \int_0^1 \phi_k(t) \log{|\tilde \sigma_\ell(s)-\tilde \sigma_k(t)|} dt, \quad s\in[0,1]
\end{equation}
where $\phi_k(t)=f(\tilde \sigma_k(t)) |\tilde \sigma_k'(t)|$.

Now, if $\ell=k$ the computation of the logarithmic kernel, when $t$ and $s$ have a relative distance of the order of the machine precision $eps$, suffers from severe loss of accuracy, because of the numerical cancellation. Then, to avoid this situation we write
$$\log{|\tilde \sigma_\ell(s)-\tilde \sigma_\ell(t)|}=\log{\frac{|\tilde \sigma_\ell(s)-\tilde \sigma_\ell(t)|}{|t-s|}}+\log{|t-s|}$$
and if $|t-s|<eps$, for the first term at the right hand side, we use the approximation
$$\log{\frac{|\tilde \sigma_\ell(s)-\tilde \sigma_\ell(t)|}{|t-s|}} \sim \log{|\tilde \sigma'_\ell(t)|}. $$
Hence, by this numerical tricks, we can rewrite (\ref{gi}) as
\begin{eqnarray}\label{gi2}
g(\tilde \sigma_\ell(s))&=&\sum_{k=1, k \neq \ell}^{n} \int_0^1 \phi_k(t) \log{|\tilde \sigma_\ell(s)-\tilde \sigma_k(t)|} dt \\ &+& \int_0^1 \phi_\ell(t) [\log{|t-s|} +\delta_\ell(t,s)] dt \nonumber
\end{eqnarray}
where
\begin{equation}
\delta_{\ell}(t,s)=\left\{\begin{array}{ll}
\log{|\tilde \sigma'_\ell(t)|}, &  |t-s|< eps, \\
\log{\frac{|\tilde \sigma_\ell(s)-\tilde \sigma_\ell(t)|}{|t-s|}}, & \mathrm{otherwise}
\end{array}\right..
\end{equation}

Now, in order to approximate the integrals appearing in (\ref{gi2}), we propose to proceed as follows
\begin{itemize}
\item for $\ell \neq k$, we use the Gauss-Legendre quadrature formula (\ref{Legendre}) obtaining

 $$ \sum_{h=1}^M \lambda_{M,h}^L \phi_k(x_{M,h}^L) \log{|\tilde \sigma_\ell(s)-\tilde \sigma_k(x_{M,h}^L)|}, $$
\item for $\ell=k$, we use a product integration rule for the first addendum getting
$$ \sum_{h=1}^M \lambda_{M,h}^L \phi_\ell (x_{M,h}^L)\sum_{\nu=0}^{M-1} c_\nu(s) p_\nu(x_{M,h}^L), $$
with $p_\nu \equiv p_\nu(v^{0,0})$ the $\nu$-th orthonormal Legendre polynomial and
$$ \displaystyle c_\nu(s)=\int_0^1 p_\nu(z) \log{|z-s|} dz$$
the $\nu$-th momentum computable by means of a recurrence formula (see, for instance, \cite{M}) and we again adopt a Gauss-Legendre quadrature formula for the second term obtaining
$$ \sum_{h=1}^M \lambda_{M,h}^L \phi_\ell (x_{M,h}^L)\delta_\ell(x_{M,h}^L,s).$$
\end{itemize}
Summarizing, we propose to approximate the right-hand side $\bar{g}$ by
$$\bar{g}_M=\left(\bar{g}_{M,1},\bar{g}_{M,2},\ldots,\bar{g}_{M,3n}\right)$$
 with
\begin{eqnarray}\label{gmi}
\bar{g}_{M,i}(s)&=&\sum_{k=1, k \neq \ell}^{n} \sum_{h=1}^M \lambda_{M,h}^L \phi_k(x_{M,h}^L) \log{|\tilde \sigma_\ell(s)-\tilde \sigma_k(x_{M,h}^L)|} \nonumber\\ & +& \sum_{h=1}^M \lambda_{M,h}^L \phi_\ell(x_{M,h}^L) \left(\sum_{\nu=0}^{M-1} c_\nu(s) p_\nu(x_{M,h}^L)+\delta_\ell(x_{M,h}^L,s) \right).
\end{eqnarray}
The following theorem establishes the corresponding error estimate.
\begin{theorem}\label{terminenoto}
Let $\Sigma\setminus \{P_1,\ldots,P_n\}$ be of class $C^{2}$ and $f \in C^p(\Sigma)$ with $p>0$. Then,  it results
\begin{equation} \label{estimateterminenoto}
\|(\bar{g}-\bar{g}_{M})(s)\|_{\infty} \leq \frac{\C}{ M}, \quad \forall s \in [0,1]
\end{equation}
where $\C \neq \C(M)$.
\end{theorem}

Let us observe that under the hyphotesis that the boundary $\Sigma\setminus \{P_1,\ldots,P_n\}$ is ($q+2$)-times differentiable, for some $q>0$, the  approximate right-hand side $\bar{g}_{M}$ tends to the exact one $\bar{g}$ with a rate of convergence of order $1/M^r$ where $r=\min\{q+1,p\}$.

We also remark that if we introduce this approximation of the right-hand side $\bar{g}$ in the proposed numerical procedure, we really solve the following system
\begin{equation}
(-\pi \mathcal{I}+\tilde {\mathcal L}_m+{\mathcal K}_m)\bar{u}_m=\bar{g}_M
\end{equation}
instead of (\ref{sistapprox}).
However in this case Theorem \ref{maintheorem} still holds true but the error estimate becomes
$$
\|(\bar{u}-\bar{u}_{m})(s)\|_{\infty} \leq  \mathcal{C}[\|(\tilde{\mathcal{L}}_m-\mathcal{L})\bar{u}(s)\|_\infty + \|(\mathcal{K}_m-\mathcal{K})\bar{u}(s)\|_\infty+ \|(\bar{g}-\bar{g}_{M})(s)\|_{\infty}].$$

\subsection{Approximation of the Neumann solution} \label{NeumannSolution}
In this subsection we propose to approximate the solution of our initial problem (\ref{Neumann}), by taking advantage of the numerical results provided by the method described in the previous subsections.

To this end we note that, according to the parametric representation (\ref{trasf_sigma}) of the arcs $\tilde \Sigma_k$
as well as to the decomposition (\ref{sigma}) and the corresponding parametrizations (\ref{trasf}),
$\forall (x,y) \in \RR^2 \setminus D$ the solution $u(x,y)$, defined in (\ref{u}), can be rewritten as
\begin{eqnarray} \label{double2}
u(x,y)&=&-\frac{1}{2 \pi}  \left\{ \sum_{k=1}^n \int_0^1 \phi_k(t) \log{|(\tilde \xi_k(t),\tilde \eta_k(t))-(x,y)|}dt \right.\nonumber \\ && \left. - \sum_{i=1}^{3n} \int_0^1 H_i(x,y,t) \bar{u}_i(t) dt  \right\},
\end{eqnarray}
where,  $\phi_k(t)=f(\tilde \sigma_k(t)) |\tilde \sigma'_k(t)|$ for $k \in \{1,\ldots,n\}$,   $\bar{u}_i(t)=u_i(\sigma_i(t))$ for  $i \in \{1,\ldots,3n\}$, and
\[
H_i(x,y,t)=\frac{\eta'_i(t)[x-\xi_i(t)]-\xi'_i(t)[y-\eta_i(t)]}{[x-\xi_i(t)]^2+[y-\eta_i(t)]^2}.
\]
We propose to approximate $u(x,y)$ by means of the function
\begin{eqnarray}\label{um}
u_m(x,y)&=&-\frac{1}{2 \pi} \left\{\sum_{k=1}^n \int_0^1  \phi_k(t) \log{|(\tilde \xi_k(t),\tilde \eta_k(t))-(x,y)|}dt   \right. \nonumber\\ & - & \left. \sum_{\scriptsize{\begin{array}{c}
i=1\\i \equiv 1,2(\mathrm{mod} \ 3)\end{array}}}^{3n}  \sum_{k=0}^{\mu_m} \lambda_{\mu_m,k}^R \ H_i(x,y,x_{\mu_m,k}^R) \bar{u}_{m,i}(x_{\mu_m,k}^R)\right. \nonumber\\ & - & \left.  \sum_{\scriptsize{\begin{array}{c}
i=1\\i \equiv 3(\mathrm{mod} \ 3)\end{array}}}^{3n}  \sum_{k=0}^{\nu_m} \lambda_{\nu_m,k}^R \ H_i(x,y,x_{\nu_m,k}^R) \bar{u}_{m,i}(x_{\nu_m,k}^R) \right\},
\end{eqnarray}
obtained by replacing in (\ref{double2}) each $\bar{u}_i$ with $\bar{u}_{m,i}$ and, then, by applying the suitable Radau quadrature formula in order to compute the integrals
\[\int_0^1 H_i(x,y,t) \bar{u}_{m,i}(t) dt. \]
Let us note that the quantities $\bar{u}_{m,i}(x_{\nu_m,k}^R)$ and $\bar{u}_{m,i}(x_{\mu_m,k}^R)$ involved in (\ref{um}) are just the solutions of the linear system (\ref{linerasist}).

Moreover, when we are not able to compute analytically the integrals
\[\int_0^1 \phi_k(t) \log{|(\tilde \xi_k(t),\tilde \eta_k(t))-(x,y)|} dt\]
in (\ref{um}), we also approximate them by means of a suitable quadrature formula.
For instance, we can use the Gauss-Legendre quadrature formula (\ref{Legendre}) with $N$ nodes and, in this way, we get the approximate solution
\begin{eqnarray}\label{umN}
u_{m,N}(x,y)&=&-\frac{1}{2 \pi} \left\{  \sum_{k=1}^n \sum_{h=1}^N \lambda_{N,h}^L \phi_k(x_{N,h}^L) \log{|(\tilde \xi_k(x_{N,h}^L),\tilde \eta_k(x_{N,h}^L))-(x,y)|} \right.\nonumber\\ & - & \left. \sum_{\scriptsize{\begin{array}{c}
i=1\\i \equiv 1,2(\mathrm{mod} \ 3)\end{array}}}^{3n}  \sum_{k=0}^{\mu_m} \lambda_{\mu_m,k}^R \ H_i(x,y,x_{\mu_m,k}^R) \bar{u}_{m,i}(x_{\mu_m,k}^R)\right. \nonumber\\ & - & \left.  \sum_{\scriptsize{\begin{array}{c}
i=1\\i \equiv 3(\mathrm{mod} \ 3)\end{array}}}^{3n}\sum_{k=0}^{\nu_m} \lambda_{\nu_m,k}^R \ H_i(x,y,x_{\nu_m,k}^R) \bar{u}_{m,i}(x_{\nu_m,k}^R) \right\},
\end{eqnarray}
The following theorem gives an error estimate for both the approximations (\ref{um}) and (\ref{umN}).
\begin{theorem} \label{harmonicerror}
Let the assumptions of Theorem \ref{maintheorem} be satisfied and let $u$ be the solution of the Neumann problem (\ref{Neumann}).
Then, denoted by $u_m$ and  $u_{m,N}$ the approximations (\ref{um}) and (\ref{umN}), respectively, $\forall (x,y) \in \RR^2 \setminus D$ we have the following pointwise error estimates
\begin{eqnarray}
|u(x,y)-u_m(x,y)| &\leq &  \frac {\C}{m}\left( \frac{1}{d^2}+\frac{1}{d} \right)+\frac{\C'}{d}\left\|\bar{u}-\bar{u}_m\right\|_{\infty}, \label{stima1}\\
|u(x,y)-u_{m,N}(x,y)| &\leq & \frac {\C}{m}\left( \frac{1}{d^2}+\frac{1}{d} \right)+\frac{\C'}{d}\left\|\bar{u}-\bar{u}_m\right\|_{\infty} + \frac{\C''}{d N}, \label{stima2}
\end{eqnarray}
where $d=\displaystyle \min_{k=1,\ldots,n}d_k$ with $d_k=\displaystyle \min_{0 \leq t \leq 1}|(x,y)-(\tilde \xi_k(t),\tilde \eta_k(t))|$. and $\C$, $\C'$ and $\C''$  are positive constants independent of $(x,y)$, $m$ and $N$.
\end{theorem}

Let us remark that the first addendum on the right-hand side of (\ref{stima1}) and (\ref{stima2}) and the last addendum on the right-hand side of (\ref{stima2}) could converge to zero with a rate greater than $1/m$ and $1/N$, respectively,  if the boundary $\Sigma\setminus \{P_1,\ldots,P_n\} \in C^{q+2}$, with $q>0$. Morover, from the previous estimates, we can deduce that the convergence becomes faster and faster as well as the point $(x,y)$ moves away from the boundary $\Sigma$.

\section{Proofs}\label{proofs}
\begin{proof}[Proof of Theorem \ref{theoremsolv}]
From well known results (see, for instance, \cite[p. 393] {A}) it follows that for any array of functions
$\rho=(\rho_1,\ldots,\rho_{3n})^T \in  {\mathcal X}$,
setting $ \varrho_i=(\rho_i, \rho_{i+1}, \rho_{i+2})^T \in C([0,1])^{^3}$,
$\forall i \in \{1,\ldots,3n\}$, $i \equiv 1 (\mathrm{mod} \ 3)$, one has
\[
{\mathcal L}^i \varrho_i \in C([0,1])^{^3}.
\]

Moreover, it is easy to see that ${\mathcal L}\rho \in {\mathcal X}$ and that if $\|\rho\|_{\infty} \leq 1$,
we have
\begin{eqnarray*}
\|{\mathcal L}\rho\|_{\infty}& = &\max_{\scriptsize{\begin{array}{c}i=1,\ldots,3n \\ i\equiv 1 (\mathrm{mod} \ 3)\end{array}}}\left\|{\mathcal L}^i \varrho_i \right\|_{\infty}
 \leq  \max_{\scriptsize{\begin{array}{c}i=1,\ldots,3n \\ i\equiv 1  (\mathrm{mod} \ 3)\end{array}}}  \left|\chi_{\frac{i-1}{3}+1}\right|\pi< \pi,
\end{eqnarray*}
from which  $\|{\mathcal L}\| < \pi.$
Then, being $\mathcal{I}:{\mathcal X} \to {\mathcal X}$ with $\|-\pi \mathcal{I}\|=\pi$,   by applying the geometric series theorem, we can deduce that the operator $(-\pi \mathcal{I}+ {\mathcal L})^{-1}: {\mathcal X} \to {\mathcal X} $  exists and is bounded
with
\[
\|(-\pi \mathcal{I}+ {\mathcal L})^{-1}\| \leq \frac{1}{\pi - \|{\mathcal L}\|}.
\]
Consequently, equation (\ref{sistop}) is equivalent to the following one
\begin{equation} \label{sistop2}
\bar{u}+(-\pi \mathcal{I}+ {\mathcal L})^{-1} {\mathcal{K}} \bar{u}=(-\pi \mathcal{I}+ {\mathcal{L}})^{-1} \bar{g}.
\end{equation}
Now, let us note that the operator $(-\pi \mathcal{I}+ {\mathcal L})^{-1} {\mathcal{K}}$ is compact since ${\mathcal{K}}$ also maps ${\mathcal X}$ into ${\mathcal X}$ and it is compact being matrix of compact operators.
Thus for equation (\ref{sistop})
the Fredholm alternative holds true and from the hypothesis it follows that system (\ref{sistop}) is unisolvent in ${\mathcal X}$ for each right-hand side $\bar{g} \in {\mathcal X}$. Finally, if $\bar{g}\in \tilde{\mathcal X}$ then $\bar{u}=(-\pi \mathcal{I}+ {\mathcal L}+\mathcal{K})^{-1}\bar{g} \in \tilde{\mathcal X}$. Indeed, since by (\ref{propr_eta}) $(-\pi \mathcal{I}+ {\mathcal L}+\mathcal{K})$ is invertible in $\mathcal{X}$, then there exists an array $\bar{\varphi}\in \tilde{\mathcal X} $ such that $\bar{\varphi}=(-\pi \mathcal{I}+ {\mathcal L}+\mathcal{K})^{-1}\bar{g}$. Hence, by the assumption we can deduce $\bar{\varphi} \equiv \bar{u}$.
\end{proof}
In order to prove Theorem \ref{operL} we need the following lemmas which can be proved by proceeding as in the proof of lemmas 2 and 3 in \cite{FL}.
\begin{lemma}\label{propLm}
Let
\[
L(t,s)=-\frac{s \sin {(\chi \pi)}}{s^2+2ts \cos {(\chi \pi)}+t^2}, \quad t,s \in [0,1],
\]for some $\chi \in \RR$, $|\chi|<-1$, and let  $e_m$   be the functional defined as in $(\ref{corR})$.
Then, for each $s \in  (0,1]$ one has
$$e_m(L(\cdot,s)) \leq \C \frac{r!}{m^r} \frac{1}{s^{r/2}},$$
where $r \in \NN$ and $\C \neq \C(m)$.
\end{lemma}

\begin{lemma}\label{densesubspace}
Let ${\mathcal X}$ be the space of functions defined in $(\ref{space})$ and let
$$\PP^{^{3n}}=\underbrace{\PP \times \PP \times \ldots \times \PP}_{3n \ \mathrm{times}} ,$$
where $\PP$ is the set of all polynomials on $[0,1]$. Then the space
\begin{equation} \label{ptilde}
\tilde{\PP}^{^{3n}}=\PP^{^{3n}}\cap {\mathcal X}
\end{equation}
is a dense subspace of ${\mathcal X}$.
\end{lemma}

\begin{proof}[Proof of Theorem \ref{operL}]
At first we note that the operators $\tilde{\mathcal L}_m$ map ${\mathcal X}$ into ${\mathcal X}$. Indeed for any $i=1,\ldots,3n$,  $i\equiv 1 (\mathrm{mod} \ 3)$ and for any array of functions
${\rho}=(\rho_1,\ldots,\rho_{3n})^T \in  {\mathcal X}$,
one has that
\[
\tilde{{\mathcal L}}_m^i \varrho_i \in C([0,1])^{^3} \quad \mathrm{and} \quad (\tilde{{\mathcal L}}_m^i \varrho_i)(0)=({\mathcal L}\varrho_i)(0),
\]
with $ \varrho_i=(\rho_i, \rho_{i+1}, \rho_{i+2})^T$.\\
Then $(\tilde {\mathcal L}_m{\rho})(0)=({\mathcal L}{\rho})(0)$ and, consequently,
$\tilde{\mathcal L}_m{\rho} \in {\mathcal X}$. \newline
Now, in order to prove (\ref{limnormWm}), we observe that
for any
${\rho}=(\rho_1,\ldots,\rho_{3n})^T \in {\mathcal X}$
such that $\|{\rho}\|_{\infty} \leq 1$ and for each $ \varrho_i=(\rho_i, \rho_{i+1}, \rho_{i+2})^T$,  we have\\

$\|\tilde{\mathcal L}_m{\rho}\|_{\infty}$
\begin{eqnarray}\label{estimatenormWm}
 &= &\hspace*{-0.8 cm} \max_{\scriptsize{\begin{array}{c}i=1,\ldots,3n \\ i\equiv 1  (\mathrm{mod} \ 3)\end{array}}} \left\|\tilde{{\mathcal L}}_m^i \varrho_i \right\|_{\infty} \nonumber \\ &= &\hspace*{-0.8 cm} \max_{\scriptsize{\begin{array}{c}i=1,\ldots,3n \\ i\equiv 1  (\mathrm{mod} \ 3)\end{array}}} \hspace{-0.6 cm} \max\left\{\sup_{s \in \left[0,\frac{c}{m^{2-2\epsilon}}\right]}\|(\tilde{{\mathcal L}}_m^i \varrho_i)(s)\|_{\infty}, \sup_{s \in\left[\frac{c}{m^{2-2\epsilon}},1\right]}\|(\tilde{{\mathcal L}}_m^i \varrho_i)(s)\|_{\infty}\right\}.
\end{eqnarray}
At this point, by repeating word by word the proof of Theorem 3 in \cite{FL}, and taking into account that the number of the quadrature points $\mu_m$ is a linear function of $m$,  one can prove that
\begin{equation} \label{supnormLm}
\|\tilde{\mathcal L}_m{\rho}\|_{\infty} \leq \pi+ \C \ \frac{r!}{m^{r \epsilon}}, \quad r \in \NN,
\end{equation}
from which we deduce (\ref{limnormWm}).
Now we prove (\ref{convpointnormWm}). To do this, we observe that from (\ref{supnormLm}) it follows that
the operators $\tilde {\mathcal L}_m:{\mathcal X} \to {\mathcal X}$ are uniformly bounded with respect to $m$, i.e.
\begin{equation} \label{assertion(ii)}
\displaystyle \sup_m \|\tilde{\mathcal L}_m \| < \infty.
\end{equation}
Moreover, by proceeding analogously to the proof of Theorem 3 in \cite{FL}, we can prove that for any ${\rho}=(\rho_1,\ldots,\rho_{3n})^T \in \tilde{\PP}^{3n}$, setting $ \varrho_i=(\rho_i, \rho_{i+1}, \rho_{i+2})^T $, one has
\begin{equation} \label{Lmi-L}
\lim_{m \to \infty} \left\| (\tilde{{\mathcal L}}_m^i- {\mathcal L}^i) \varrho_i \right\|_\infty = 0, \quad \forall i \in \{1,\ldots,3n\},  \ i \equiv 1(\mathrm{mod} \ 3).
\end{equation}
Hence, taking into account Lemma \ref{densesubspace}, we can deduce (\ref{convpointnormWm}) by applying the Banach-Steinhaus theorem (see, for instance, \cite[p. 517]{A}). The proof is complete.
\end{proof}

\begin{proof}[Proof of Theorem \ref{opercompatti}]
At first let us note that the operators ${\mathcal K}_m$ map ${\mathcal X}$ into ${\mathcal X}$ and the set $\{{\mathcal K}_m\}_m$ is collectively compact if so is the set  $\{\mathcal{K}_m^{i,j}\}_m$ for any $i,j \equiv 1,2 (\mathrm{mod} \ 3)$ with $|i-j|\neq 1$ or $i \equiv 0 (\mathrm{mod} \ 3)$ and the set $\{\mathcal{M}_m^{i,j}\}_m$ for any $i,j \equiv 1,2 (\mathrm{mod} \ 3)$ with $|i-j|= 1$. Moreover, by  definition, it results that $\forall  \rho=(\rho_1,\ldots,\rho_{3n}) \in {\mathcal X}$,
 if
\begin{equation}\label{convpunt1}
\lim_{m \to \infty} \|(\mathcal{K}_m^{i,j}-\mathcal{K}^{i,j}) \rho_j\|_\infty=0
\end{equation}
when $i,j \equiv 1,2(\mathrm{mod} \ 3)$  with $|i-j|\neq 1$  or $i \equiv 0(\mathrm{mod} \ 3)$, and
\begin{equation}\label{convpunt2}
\lim_{m \to \infty} \|(\mathcal{M}_m^{i,j}-\mathcal{M}^{i,j}) \rho_j\|_\infty=0
\end{equation}
when $i,j \equiv 1,2(\mathrm{mod} \ 3)$  with $|i-j|= 1$,
then $$\displaystyle \lim_{m \to \infty} \|({\mathcal K}_m -{\mathcal K}){\rho}\|_{\infty}=0.$$

Now, the limit conditions (\ref{convpunt1}) and (\ref{convpunt2}) can be immediately deduced ta\-king into account the  definitions (\ref{Kmij1}), (\ref{Kmij2}) and (\ref{Mmij}) of the finite rank ope\-ra\-tors $\mathcal{K}_m^{i,j}$ and  $\mathcal{M}_m^{i,j}$, the continuity of the kernels $K^{i,j}(\cdot,s)$ and $M^{i,j}(\cdot,s)$ and the convergence of the Radau quadrature rule on the set $C([0,1])$. \newline
Then,  using standard arguments (see, for instance, \cite[Theorem 12.8]{K}), it also follows that the sets $\{\mathcal{K}_m^{i,j}\}_m$ and $\{\mathcal{M}_m^{i,j}\}_m$, with $i$ and $j$ as above, are collectively compact and the proof is complete.
\end{proof}

\begin{proof}[Proof of Theorem \ref{maintheorem}]
From Theorem \ref{operL} we can deduce that the operators $-\pi \mathcal{I}+\tilde{\mathcal L}_m:{\mathcal X} \to :{\mathcal X}$ are bounded and pointwise convergent to $-\pi \mathcal{I}+{\mathcal L}$. Moreover, in virtue of the geometric series theorem, it follows  that for sufficiently large $m$ the operators $(-\pi \mathcal{I}+\tilde{\mathcal L}_m)^{-1}:{\mathcal X} \to {\mathcal X}$ exist and are uniformly bounded with
$$\|(-\pi \mathcal{I}+\tilde{\mathcal L}_m)^{-1}\| \leq \frac{1}{\pi-\displaystyle \sup_m\|\tilde{\mathcal L}_m\|}.$$ Consequently, taking also into account Theorem \ref{opercompatti}, it results (see, for instance, Theorem 10.8 and Problem 10.3 in \cite{K}) that for  sufficiently large $m$ the
operators  $$(-\pi \mathcal{I}+\tilde{\mathcal L}_m +{\mathcal K}_m)^{-1}:{\mathcal X} \to {\mathcal X}$$ exist and are uniformly bounded, i.e. the method is stable. Now,
by using the same arguments as in the proof of theorems 5 and 6 in \cite{FL}, one can assert (\ref{conditionnumber}) and finally show that the error estimate (\ref{errorestimate1}) holds true.
\end{proof}
In order to prove  Theorem \ref{terminenoto}, we recall the definition of the error of best polynomial approximation in uniform norm for a
function $f \in C([0,1])$
$$E_m(f)_\infty=\inf_{P_m \in \PP_m}\|(f-P_m)\|_\infty, $$
and the following error estimate for the Gauss-Legendre quadrature formula (\ref{Legendre})
\begin{equation}\label{uniformerror}
|e_m(f)| \leq \C E_{2m-1}(f), \quad \C \neq \C(m,f).
\end{equation}
Moreover we mention that for $f,g \in C([0,1])$ it results
\begin{equation}\label{errprod}
E_{2m-1}(fg)_\infty \leq 2 \|f\|_\infty E_{[\frac{2m-1}{2}]}(g)_\infty+E_{[\frac{2m-1}{2}]}(f)_\infty \|g\|_\infty.
\end{equation}

\begin{proof}[Proof of Theorem \ref{terminenoto}]
Since
$$\|(\bar{g}-\bar{g}_{M})(s)\|_\infty =\max_{i=1,\dots,3n}|\bar{g}_i(s)-\bar{g}_{M,i}(s)|,$$
in order to prove (\ref{estimateterminenoto}), we are going to estimate the $i$-th term $|\bar{g}_i(s)-\bar{g}_{M,i}(s)|$.
By (\ref{gsigma}), (\ref{gi2}) and (\ref{gmi}), it results
\begin{eqnarray}\label{g-gm}
|\bar{g}_i(s)-\bar{g}_{M,i}(s)| \hspace*{-0.5cm}&\leq &  \hspace*{-1cm} \sum_{\scriptsize{\begin{array}{c}k=1 \\ k\neq \ell \ \end{array}}}^{n} \left| \int_0^1 \phi_k(t) \log{|\tilde \sigma_k(s_i)-\tilde \sigma_\ell(t)|} dt  \right. \nonumber \\ &\hspace{1 cm}& \hspace*{-1cm}-\left.  \sum_{h=1}^M \lambda_{M,h}^L \phi_k(x_{M,h}^L) \log{|\tilde \sigma_k(s_i)-\sigma_\ell(x_{M,h}^L)|}  \right| \nonumber\\
\hspace*{-6cm}+ &  & \hspace*{-1.5cm}\left| \int_0^1 \phi_\ell(t) \log{|t-s_i|}  dt - \sum_{h=1}^M \lambda_{M,h}^L \phi_\ell(x_{M,h}^L)  \sum_{\nu=0}^{M-1} c_\nu(s_i) p_\nu(x_{M,h}^L)\right| \nonumber\\
\hspace*{-6cm}+ &  & \hspace*{-1.5cm} \left| \int_0^1 \phi_\ell(t) \delta_\ell(t,s_i) dt- \sum_{h=1}^M \lambda_{M,h}^L \phi_\ell(x_{M,h}^L) \delta_\ell(x_{M,h}^L,s_i) \right| \nonumber \\
\hspace*{-6cm} =:&  & \hspace*{-1.5cm}  \sum_{\scriptsize{\begin{array}{c}k=1 \\ k\neq \ell \end{array}}}^{n}\left| A_{k,\ell}(s_i)\right|+|B_\ell(s_i)|+|C_\ell(s_i)| \ .
\end{eqnarray}
Now let us consider $A_{k,\ell}(s_i)$, for $k$ fixed. By (\ref{uniformerror}) and (\ref{errprod}),  one can write
\begin{eqnarray}\label{Aij}
|A_{k,\ell}(s_i)| &\leq &  \C E_{2M-1}(\phi_k \log{|\tilde \sigma_k(s_i)-\tilde \sigma_\ell(\cdot)|})_\infty \nonumber \\  &\leq & \C  \left( E_{[\frac{2M-1}{2}]}(\phi_k)_\infty \|\log{|\tilde \sigma_k(s_i)-\tilde \sigma_\ell(\cdot)|}\|_\infty \right. \nonumber \\ & \hspace{1 cm}& +\left. 2 \|\phi_k\|_\infty E_{[\frac{2M-1}{2}]}(\log{|\tilde \sigma_k(s_i)-\tilde \sigma_\ell(\cdot)|})_\infty \right) \nonumber
\\ &\leq &  \frac{\C}{M},
\end{eqnarray}
being, by the assumptions, $\phi_k \in C^{1}$ and
$\log{|\tilde \sigma_k(s_i)-\tilde \sigma_\ell(\cdot)|} \in C^{2}$ for each $s_i \in [0,1]$.
Using the same arguments, one has that
\begin{equation}\label{Ci}
|C_\ell(s_i)| \leq \frac{\C}{M}
\end{equation}
being, under the hypotheses, $\phi_\ell \in C^{1}$ and $\delta_\ell(\cdot,s_i) \in C^{1}$ for each $s_i \in [0,1]$ and for each $\ell=1,\,\dots,\,n$.
Finally, by applying \cite[(3.11)]{MS98}, it results
\begin{equation}\label{Bi}
|B_\ell(s_i)| \leq \frac{\C}{M}.
\end{equation}
Hence,  by using (\ref{Aij}), (\ref{Ci}) and (\ref{Bi})  in (\ref{g-gm}), the thesis follows.
\end{proof}

\begin{proof}[Proof of Theorem \ref{harmonicerror}]
Estimate (\ref{stima1}) can be proved by proceeding as in the proof of Theorem 7 in \cite{FL}, taking also into account that the number of the involved quadrature nodes is a linear function of $m$. Concerning inequality (\ref{stima2}), we note that, by definition, it results
\begin{equation}\label{u-umN}
|u(x,y)-u_{m,N}(x,y)| \leq |u(x,y)-u_{m}(x,y)|+|R_N(x,y)|,
\end{equation}
where
\begin{eqnarray}
 R_N(x,y)&=& \sum_{k=1}^{n}  \left(\int_0^1 \phi_k(t) \log{|(\tilde \xi_k(t),\tilde \eta_k(t))-(x,y)|} dt \right. \nonumber \\ &\hspace{0cm} &\left. -\sum_{h=1}^N \lambda_{M,h}^L \phi_k(x_{M,h}^L) \log{|(\tilde \xi_k(x_{M,h}^L),\tilde \eta_k(x_{M,h}^L))-(x,y)|} \right)\\ & =:& \sum_{k=1}^{n} R_{N,k}(x,y).
\end{eqnarray}
Then, taking into account (\ref{corL}), we can write
\begin{eqnarray}
 \left|R_{N,k}(x,y)\right|& \leq & \frac{\C}{N} E_{2N-2}((\phi_k(\cdot) \log{|(\tilde \xi_k(\cdot),\tilde \eta_k(\cdot))-(x,y)|})^{'})_{\varphi,1} \nonumber \\ & \leq & \frac{\C}{N} \|(\phi_k(\cdot) \log{|(\tilde \xi_k(\cdot),\tilde \eta_k(\cdot))-(x,y)|})^{'} \varphi\|_1 \nonumber
\\ & \leq & \frac{\C}{N} \int_0^1 \sum_{j=0}^1
  (\log{|(\tilde \xi_k(t),\tilde \eta_k(t))-(x,y)|})^{(j)} \phi_k^{(1-j)}(t) t^{\frac{1}{2}} dt \nonumber \\ & \leq & \frac{\C}{ N} \sum_{j=0}^1
\end{eqnarray}
where $d=\displaystyle \min_{k=1,\ldots,n}d_k$ with $d_k=\displaystyle \min_{0 \leq t \leq 1}|(x,y)-(\tilde \xi_k(t),\tilde \eta_k(t))|$.
Hence, by using (\ref{stima1}) and (\ref{RN}) in (\ref{u-umN}), we get (\ref{stima2}).
\end{proof}

\section{Numerical Tests}\label{tests}
In this section we apply the  method described in Section \ref{the method} for the  nu\-me\-rical solution of some examples of the exterior Neumann problem on planar domains with corners. \newline
In each test, in order to give the boundary condition $f$, we choose  a test harmonic function $u$ and we perform the absolute error $\varepsilon_{m,N}(x,y)=|u(x,y)-u_{m,N}(x,y)|$ at the point $(x,y)\in \RR^2 \setminus D$ where $u_{m,N}$ is as in (\ref{umN}). Moreover, we also analyze the condition number in infinity norm of the matrix $A_m$ of the linear system (\ref{linearsystem2}). \newline
All the numerical results are obtained by approximating the right-hand sides by using (\ref{gmi}) with  $M=\frac{\nu_m}{2}$.

{\bf{Example 1.}}
\begin{figure}[!t]
\centering
\includegraphics[scale=0.5]{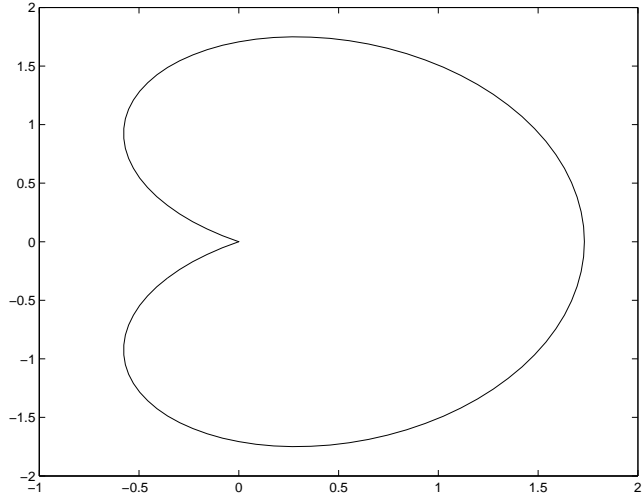}
\caption{The contour $\Sigma$ in Example 1 with $\phi=\frac{5}{3}\pi$
\label{bound_heart}}
\end{figure}
Let us consider a family of ``heart-shaped" domains (see Figure \ref{bound_heart}) bounded by the curves
$$\sigma(t)=\left(\begin{array}{c}\cos{(1+\frac{\phi}{\pi})\pi t}-\sin{(1+\frac{\phi}{\pi})\pi t}\\
\sin{(1+\frac{\phi}{\pi})\pi t}+\cos{(1+\frac{\phi}{\pi})\pi t}\end{array}\right)\left(\begin{array}{c}\tan{\frac{\phi}{2}}\\ 1\end{array}\right)
-\left(\begin{array}{c} \tan{\frac{\phi}{2}}\\ \cos{\pi t}\end{array}\right), \quad t \in [0,1], $$
where $\phi \in (\pi,2\pi)$ is the interior angle of the single outward-pointing corner $P_1=(0,0)$ and let us decompose the given boundary as shown in Subsection \ref{formulation} with $\delta$ in (\ref{delta1}) and (\ref{delta2}) given by  $\delta=3.87e-07$.
For this test we choose boundary data corresponding to the exact solution
$$u(P)=\log{|P-Q_1|}-\log{|P-Q_2|}, \quad Q_1=(0.5,0), \quad Q_2=(0.2,0).$$
Table \ref{table_heart} reports the results obtained by applying our method for  $\phi=\frac 5 3 \pi$ and $N=\frac{\nu_m}{2}$ while Figure \ref{cond_heart} shows the condition number in infinity norm of the matrix $A_m$ as a function of the interior angle $\phi$, confirming that
the estimate (\ref{conditionnumber}) holds true whatever the angle at the corner point.

\begin{table}[htb]
\caption{Numerical results for Example 1 with $\phi=\frac{5}{3}\pi$, $c=300$ and $\epsilon=10^{-3}$ \label{table_heart}}
{
\begin{tabular}{|c|c|c|c|c|c|c|c|}
\hline
$\mu_m$  & $\nu_m$ & $\varepsilon_{m,N}(-0.1,0)$ &  $\varepsilon_{m,N}(3,3)$ &$\varepsilon_{m,N}(-40,-50)$ & $\varepsilon_{m,N}(100,-100)$& $\mathrm{cond}(A_m)$ \\
\hline
8 & 32   &    6.62e-03  &  2.35e-05  &  4.52e-05  &  5.9e-05& 133.5\\ \hline
16 & 64   &   6.95e-03  &  1.89e-04  &  1.12e-05  &  5.3e-06& 25.86\\ \hline
32 & 128  &   6.78e-04  &  1.81e-05  &  1.05e-06  &  5.3e-07& 18.37\\ \hline
64 & 256 &    1.18e-05  &  3.19e-07  &  1.86e-08  &  9.2e-09& 18.32 \\ \hline
128 & 512 &   2.29e-06  &  6.10e-08  &  3.55e-09  &  1.8e-09& 18.32   \\ \hline
\end{tabular}
}
\end{table}

\begin{figure}[!t]
\centering
\includegraphics[scale=0.5]{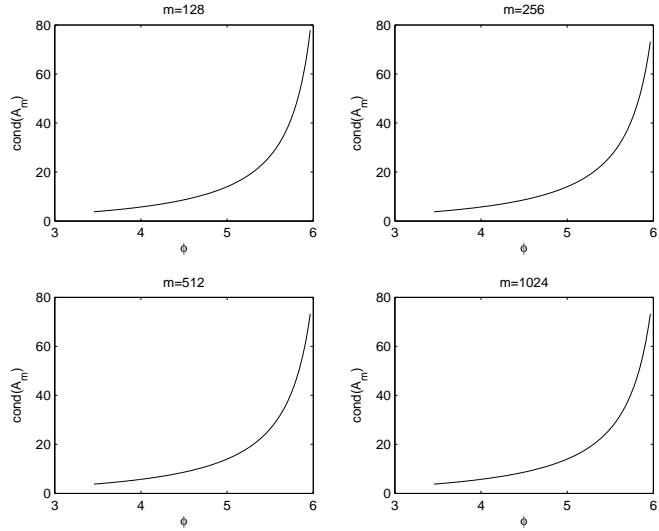}
\caption{Condition Numbers for  Example 1 with $c=300$ and $\epsilon=10^{-3}$
\label{cond_heart}}
\end{figure}

{\bf Example 2.}
Let us consider a family of ``teardrop" domains (see Figure \ref{bound_teardrop}) bounded by the curves parameterized by
$$\sigma(t)= \left(2\sin{\pi t},-\tan{\frac{\phi}{2}}\sin{2\pi t}\right), \quad t\in[0,1],$$
where $\phi \in (0,\pi)$ is the interior angle of the single outward-pointing corner $P_1=(0,0)$ and let us choose the boundary data $f$ as the normal derivative of the following function
$$ u(x,y)=\arctan{\left(\frac{y-0.2}{x-0.8}\right)}-\arctan{\left(\frac{y}{x-0.8} \right)}. $$
Moreover, for this text we consider $\delta=5.37e-011$ in (\ref{delta1}) and (\ref{delta2}).
\begin{figure}[!t]
\centering
\includegraphics[scale=0.5]{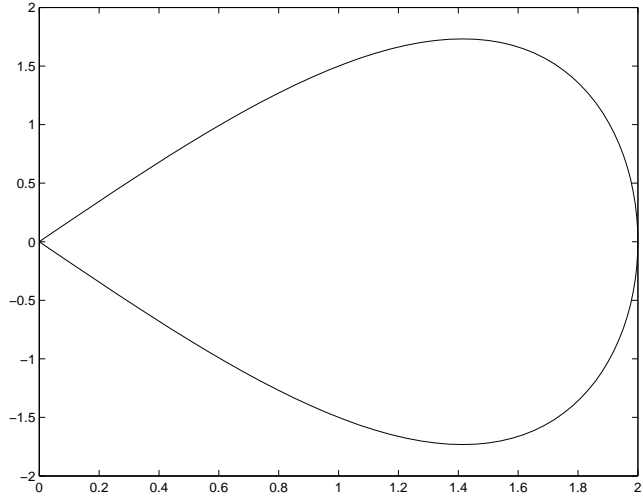}
\caption{The contour $\Sigma$ in Example 2 with $\phi=\frac{2}{3}\pi$
\label{bound_teardrop}}
\end{figure}
The numerical results reported in  Table \ref{table_teardrop1}, obtained for $\phi=\frac{2}{3}\pi$ and $N=\frac{\nu_m}{2}$, agree with the theoretical estimate (\ref{errorestimate1}). Moreover, Figure \ref{cond_teardrop} confirms that the sequence $\{cond(A_m)\}_{m \geq m_0}$ is uniformly bounded with respect to $m$, according to the theoretical estimate (\ref{conditionnumber}).

\begin{table}[htb]
\caption{Numerical results for Example 2 with $\phi=\frac{2}{3}\pi$,
 $c=100$ and $\epsilon=10^{-3}$ \label{table_teardrop1}}
{
\begin{tabular}{|c|c|c|c|c|c|c|c|}
\hline
$\mu_m$  & $\nu_m$ & $\varepsilon_{m,N}(-0.1,0)$ &  $\varepsilon_{m,N}(3,3)$ &$\varepsilon_{m,N}(-40,-50)$ & $\varepsilon_{m,N}(100,-100)$& $\mathrm{cond}(A_m)$ \\
\hline
8 & 32   & 1.44e-03  &  6.39e-04  &  1.47e-03 &   1.80e-03& 6.67 \\ \hline
16 & 64  & 7.43e-06  &  8.81e-06  &  4.34e-06 &   5.57e-06& 4.49\\ \hline
32 & 128 & 9.32e-08  &  2.54e-07  &  1.98e-08 &   8.05e-09& 4.16\\ \hline
64 & 256 & 8.24e-08  &  1.03e-08  &  8.14e-10 &   3.62e-10& 4.16\\ \hline
128 & 512& 2.14e-08  &  2.92e-09  &  2.29e-10 &   1.01e-10& 4.17 \\ \hline
\end{tabular}
}
\end{table}

\begin{figure}[!t]
\centering
\includegraphics[scale=0.5]{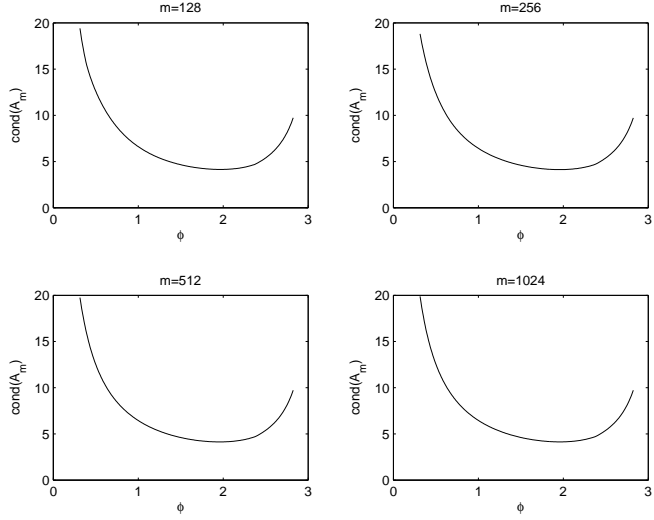}
\caption{Condition Numbers for  Example 2 with $c=100$ and $\epsilon=10^{-3}$
\label{cond_teardrop}}
\end{figure}

{\bf Example 3.}
Let us consider a family of ``boomerang" domains (see Figure \ref{bound_boomerang}) having as boundaries the following curves
$$\sigma(t)= \left(\frac{2}{3}\sin{3\pi t},-\tan{\frac{\phi}{2}}\sin{2\pi t}\right), \quad t\in[0,1],$$
where $\phi \in (\pi,2\pi)$ is the interior angle of the single inward-pointing corner $P_1=(0,0)$. Table \ref{table_boomerang} shows the numerical results obtained in the case where the boundary data $f$ is the normal derivative of the  function
$$u(P)=\log{|P-Q_1|}-\log{|P-Q_2|}, \quad Q_1=(-0.1,0), \  Q_2=(-0.2,0),$$ the interior angle at $P_1$ is $\phi=\frac{3}{2}\pi$, $\delta=5.16e-08$ in (\ref{delta1}) and (\ref{delta2}) and $N=\nu_m$.
\begin{figure}[!t]
\centering
\includegraphics[scale=0.5]{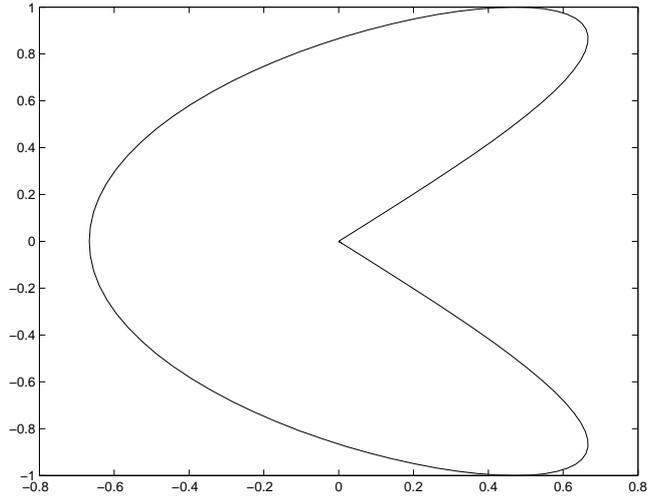}
\caption{The contour $\Sigma$ in Example 3 with $\phi=\frac{3}{2}\pi$
\label{bound_boomerang}}
\end{figure}

\begin{table}[htb]
\caption{Numerical results for Example 3 with $\phi=\frac{3}{2}\pi$,
 $c=100$ and $\epsilon=10^{-3}$\label{table_boomerang} }
{
\begin{tabular}{|c|c|c|c|c|c|c|c|}
\hline
$\mu_m$  & $\nu_m$ & $\varepsilon_{m,N}(0.2,0)$ &  $\varepsilon_{m,N}(3,3)$ &$\varepsilon_{m,N}(-40,-50)$ & $\varepsilon_{m,N}(100,-100)$& $\mathrm{cond}(A_m)$ \\
\hline
8 & 32    &     7.22e-03 &   2.66e-04 &   7.41e-06  &  3.48e-05 & 19.13\\ \hline
16 & 64   &    3.34e-04  &  1.62e-05  &  9.59e-07   & 4.87e-07  & 16.92\\ \hline
32 & 128  &    8.51e-05  &  4.05e-06  &  2.39e-07   & 1.21e-07  & 16.92\\ \hline
64 & 256  &     1.95e-05 &   9.34e-07 &   5.54e-08  &  2.80e-08 & 16.93\\ \hline
128 & 512 &    4.64e-06  &  2.22e-07  &  1.31e-08   & 6.67e-09  & 16.93\\ \hline
\end{tabular}
}
\end{table}

\begin{figure}[!t]
\centering
\includegraphics[scale=0.5]{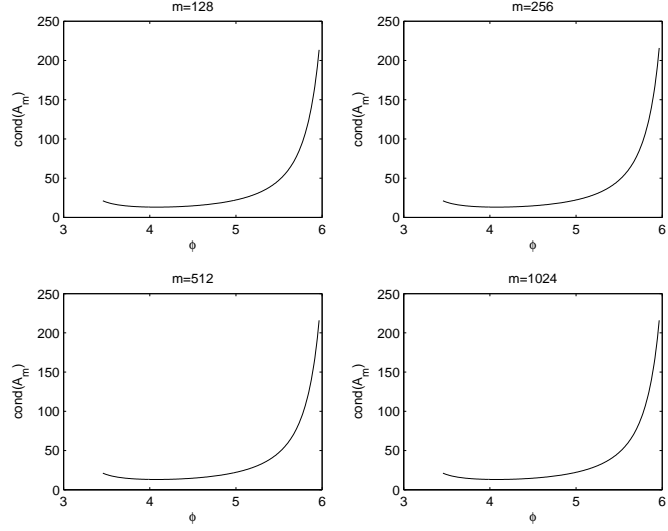}
\caption{Condition Numbers for  Example 3 with $c=300$ and $\epsilon=10^{-3}$
\label{cond_boomerang}}
\end{figure}

{\bf Example 4.}
Let $D$ be the polygonal domain represented in Figure \ref{bound_triangolo} with vertices $(-5/4,-3/4)$, $(3/4,-3/4)$ and $(3/4, 5/4)$ and apply the method described in Section \ref{the method}  in the case when the exact solution of (\ref{Neumann}) is the following harmonic function
$$u(x,y)=\frac{x^2-y^2}{(x^2+y^2)^2}.$$
\begin{figure}[!t]
\centering
\includegraphics[scale=0.5]{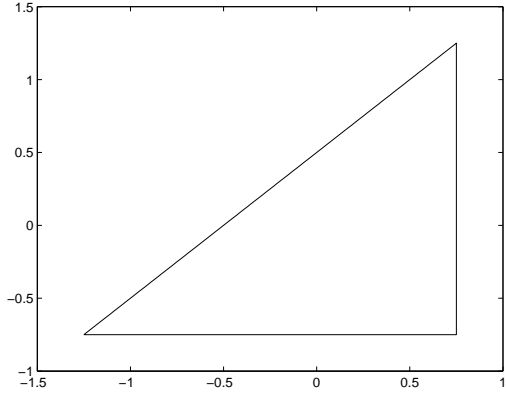}
\caption{The contour $\Sigma$ in Example 4
\label{bound_triangolo}}
\end{figure}
Table \ref{table_triangolo} contains the numerical results obtained with $N=\frac{\nu_m}{2}$.

\begin{table}[htb]
\caption{Numerical results for Example 4 with $c=100$ and $\epsilon=10^{-6}$ \label{table_triangolo} }
{
\begin{tabular}{|c|c|c|c|c|c|c|c|}
\hline
$\mu_m$  & $\nu_m$ & $\varepsilon_{m,N}(-1.5,1.5)$ &  $\varepsilon_{m,N}(2,2)$ &$\varepsilon_{m,N}(10,20)$ & $\varepsilon_{m,N}(100,100)$& $\mathrm{cond}(A_m)$ \\
\hline
8 & 32  & 7.11e-04 &  1.33e-02 & 1.79e-03 & 2.93e-04 & 166.39 \\ \hline
16 & 64   & 9.83e-04 &  2.78e-04 & 6.51e-05 & 6.70e-06 &  66.18 \\ \hline
32 & 128 & 1.41e-04 & 7.74e-06 & 6.94e-06 & 5.27e-07 & 20.40 \\ \hline
64 & 256  & 2.18e-06  & 3.38e-07 & 1.24e-07 &  1.12e-08 & 9.11 \\ \hline
128 & 512 & 6.93e-09   & 1.20e-09  & 4.16e-10 & 3.90e-11 & 8.81 \\ \hline
\end{tabular}
}
\end{table}

\textbf{Remarks}

The numerical results, shown in this section, confirm the theoretical ones stated in Section \ref{the method}.
We can note that,  according to estimate (\ref{stima2}), for any fixed $m$, the error $\varepsilon_{m,N}(x,y)$, becomes smaller and smaller as well as the distance of the exterior point $(x,y)$ from the boundary is larger and larger. \newline
Moreover, the results put in evidence that, as stated in (\ref{conditionnumber}), the sequence $\{\mathrm{cond}(A_m)\}_{m \geq m_0}$ is uniformly bounded with respect to $m$. \newline
The computational cost of the proposed procedure of course grows with the number of corners of the domain. However, when
the requested precision is not too high (as it is usual in the applications), the dimension of the linear system
(\ref{linearsystem2}) is kept down.  Anyway, such system is still well conditioned also when
its dimension is larger, whatever the interior angles at the corner points.

\section*{Acknowledgments}
C. Laurita is partly supported by Istituto Nazionale di Alta Matematica, GNCS Project 2013 ``Metodi fast per la risoluzione numerica di sistemi di equazioni integro-differenziali''.

\bibliographystyle{plain}
\bibliography{ANM-refs}
\end{document}